\documentclass[a4, 12pt]{amsart}

\usepackage{fullpage}
\usepackage[title]{appendix}
\usepackage{graphicx}
\usepackage{subcaption}
\usepackage{amsthm}
\usepackage{lmodern}
\usepackage{mathrsfs} 
\usepackage{amsfonts}
\usepackage{amssymb}
\usepackage{diagbox}
\usepackage{mathtools}
\usepackage{longtable} 
\usepackage{MnSymbol}
\usepackage{indentfirst}
\usepackage{listings,xcolor}
\usepackage[all]{xy}
\usepackage[colorlinks=true,linkcolor=blue]{hyperref}
\usepackage{ytableau}
\usepackage{tikz-cd}
\usepackage{enumerate}
\theoremstyle{plain}
\newtheorem{theorem}{Theorem}[section]

\newtheorem{lemma}[theorem]{Lemma}

\newtheorem{proposition}[theorem]{Proposition}

\theoremstyle{remark}

\newtheorem{remark}{Remark}[section]

\makeatletter

\newcommand{\Rmnum}[1]{\expandafter\@slowromancap\romannumeral #1@}
\makeatother

\def\bz{\mathbb Z}
\def\bc{{\mathbb C}}

\newmuskip\pFqmuskip

\numberwithin{equation}{section}
\allowdisplaybreaks[2] 

\begin{document}

\title[Unimodality and concave composition]{Monotonicity of the rank functions for \\ concave compositions}
\author{Nian Hong Zhou}

\address{N. H. Zhou: School of Mathematics and Statistics, The Center for Applied Mathematics of Guangxi, Guangxi Normal University, Guilin 541004, Guangxi, PR China}
\email{nianhongzhou@outlook.com; nianhongzhou@gxnu.edu.cn}%


\begin{abstract}
A (strongly) concave composition of an integer $n$ is a sequence of positive integers that is (strictly) decreasing to a point and then (strictly) increasing thereafter, such that the sum of the entries equals $n$. The value at the low point is called the center part. The difference between the number of entries before and after the low point of the sequence is referred to as the rank of the (strongly) concave composition. The rank functions $V_d(m,n)$ and $V(m,n)$ are defined as the number of concave compositions and strongly concave compositions, respectively, of $n$ with rank $m$. By constructing the difference systems that characterize the rank generating functions, we establish monotonicity properties for the rank functions of both strongly concave compositions and concave compositions for all positive integers $n$. Moreover, we also study the monotonicity properties for the rank functions of (strongly) concave compositions with fixed center parts.
\end{abstract}
\maketitle

\section{Introduction}
A \emph{composition} of an integer $n$ is defined as an ordered sequence of positive integers summing to $n$.
The study of compositions has a long history, dating back to MacMahon's seminal work.
It is natural to impose restrictions on the ascents or descents of consecutive parts within a composition.
For instance, compositions with no ascents between consecutive parts correspond precisely to the well--known \emph{integer partitions}. In \cite{MR3048655}, Andrews examines a new restriction on compositions, termed \emph{concave compositions}. A concave composition $\sigma$ is a nonnegative integer sequence of the form
\begin{align}\label{eq1}
a_1\ge a_2\ge \dots\ge a_{l}>c<b_{r}\le \dots\le b_{2}\le b_1,
\end{align}
where $\ell_L(\sigma):=l$, $\ell_R(\sigma):=r$, $c(\sigma):=c$ and 
$$|\sigma|:=a_1+\dots+a_{l}+c+b_{r}+ \dots+b_1$$ are called the number of parts in the left hand-side, the number of parts in the right hand-side, the central part and the size of $\sigma$, respectively. If all the ``$\le$" and ``$\ge$" in \eqref{eq1} are replaced by ``$<$" and ``$>$", respectively, we refer to a \emph{strongly
concave composition}. We call $\sigma$ a (strongly) concave composition of an integer $n$ if $|\sigma|=n$.

\medskip

Denote by $V(n)$ and $V_d(n)$ the number of all concave compositions and strongly concave compositions, respectively, of $n$. Andrews \cite{MR3048655} proved that
$$
v(q):=\sum_{n\ge 0}V(n)q^n=\sum_{n\ge 0}\frac{q^n}{(q^{n+1})_{\infty}^2}
$$
and
$$
v_d(q):=\sum_{n\ge 0}V_d(n)q^n=\sum_{n\ge 0}(-q^{n+1})_{\infty}^2q^n.
$$
Here $(a)_{n}:=(a;q)_n$ is the $q$-shifted factorial defined by
$$(a;q)_n:=\prod_{0\le j<n}(1-aq^j)$$
for any $n\in \bz_{\ge 0}\cup\{\infty\}$. The generating functions $v(q)$ and $v_d(q)$ exhibit elegant modular  properties. In particular, Andrews \cite{MR3048655} proved that
$v_d(q)+\sum_{n\ge 0}(-1)^nq^{{n(n+1)}/{2}}$ is essentially a \emph{modular form} multiplied by a \emph{partial theta function}, while Andrews--Rhoades--Zwegers \cite{MR3152010} demonstrated that $v(q)$ is a \emph{mixed mock modular form} in a more general sense than is typically used.
\medskip

By definition, we see that the (strongly) concave composition statistics $\ell_L$ and $\ell_R$ are equidistributed.
To measure and characterize the position of the central part $c(\sigma)$, Andrews--Rhoades--Zwegers \cite{MR3152010} define a statistic called the \emph{rank} of the concave composition $\sigma$ as follows:
\begin{align*}
{\rm rank}(\sigma):=\ell_{R}(\sigma)-\ell_L(\sigma).
\end{align*}
Denote by $V^c(m,n)$ and $V_d^c(m,n)$ the number of concave compositions and strongly concave compositions, respectively, of $n$ with rank $m$ and central part $c$. Let
$$V(m,n):=\sum_{c\ge 0}V^c(m,n)\;\; \text{and}\;\; V_d(m,n):=\sum_{c\ge 0}V_d^c(m,n),$$
which counts the number of all concave compositions and strongly concave compositions, respectively, of $n$ with rank $m$. Then, by standard counting techniques, we obtain the following generating functions:
\begin{align*}
\mathcal{V}^c(z,q)&:=\sum_{n\ge 0}\sum_{m\in\bz}V^c(m,n)z^{m}q^n=\frac{q^c}{(zq^{c+1})_\infty (z^{-1}q^{c+1})_\infty},\\
\mathcal{V}(z,q)&:=\sum_{n\ge 0}\sum_{m\in\bz}V(m,n)z^{m}q^n=\sum_{c\ge 0}\frac{q^c}{(zq^{c+1})_\infty (z^{-1}q^{c+1})_\infty},\\
\mathcal{V}_d^c(z,q)&:=\sum_{n\ge 0}\sum_{m\in\bz}V_d^c(m,n)z^{m}q^n=q^c(-zq^{c+1})_\infty (-z^{-1}q^{c+1})_\infty,\\
\mathcal{V}_d(z,q)&:=\sum_{n\ge 0}\sum_{m\in\bz}V_d(m,n)z^{m}q^n=\sum_{c\ge 0}q^c(-zq^{c+1})_\infty (-z^{-1}q^{c+1})_\infty.
\end{align*}
Clearly, for any $n\ge 0$, all of $V^c(m,n)$, $V(m,n)$, $V_d^c(m,n)$ and $V_d(m,n)$ are even functions in $m$ on $\bz$. The above generating functions also exhibit elegant modular properties. 
For examples, from Andrews--Rhoades--Zwegers \cite[Theorem 1.6]{MR3152010}, we see that
the generating functions $\mathcal{V}(z,q)$ and $\mathcal{V}^c(z,q)$ are related to \emph{quantum Jacobi forms} and \emph{mock Jacobi forms}.

\medskip

Recall that a \emph{(strictly) unimodal sequence} is a finite sequence of real numbers that first (strictly) increases and then (strictly) decreases. In \cite[Theorem 1.3]{MR4001540}, Zhou established that
uniformly for all integers $m=o\left(n^{3/8}\right)$ and as $n\to\infty$,
\begin{equation*}
\frac{V_d\left(m,n\right)}{V_d(n)}\sim \frac{1}{\sqrt{2\pi}}\beta_n^{1/2}e^{-\frac{1}{2} \beta_n m^2}.
\end{equation*}
This uniform asymptotic formula indicated that under suitably normalized, the rank function $V_d(m,n)$ of a strongly concave compositions have a Gaussian distribution as limit law, i.e., for large $n$ plotting  $V_d(m,n)$ against rank $m$ should give the bell curve of a Gaussian distribution. This was actually conjectured by Andrews--Rhoades--Zwegers \cite[p. 2109]{MR3152010}. This Gaussian limit law also suggests the unimodality of $V_d(m,n)$. Recently, by using an algebraic method, Zhou \cite[Theorem 4.3]{MR4903333} proved that the sequences:
$$\big(V(2m,n)\big)_{m\in\bz}\;\; \text{and}\;\; \big(V(2m+1,n)\big)_{m\in\bz}$$
are unimodal for any $n \ge 0$.  In \cite[Theorem 4.3]{MR4903333},  Zhou also proved that the sequences:
$$\big(V_d(2m,n)\big)_{m\in\bz}\;\; \text{and}\;\; \big(V_d(2m+1,n)\big)_{m\in\bz}$$
are strictly unimodal for any $n \ge 0$.  Zhou \cite[Conjecture 4.1]{MR4903333} further conjectured that the sequence $(V(m,n))_{m\in\bz}$ and $(V_d(m,n))_{m\in\bz}$ is unimodal for any $n \ge 0$.

\medskip

In this paper, we establish some monotonicity results for the rank functions $V^c(m,n)$, $V(m,n)$, $V_d^c(m,n)$ and $V_d(m,n)$. In particular, we refine and prove the above unimodality conjecture of the author for $(V_d(m,n))_{m\in\bz}$ and give a partial answer for the unimodality of $(V(m,n))_{m\in\bz}$.  Throughout this paper, for any power series $\sum_{n\ge 0}A(m,n)q^n$ where $m$ is a nonnegative integer, we define for $m\ge 1$:
\begin{align}\label{eqda}
D_{A}(m, q):=\sum_{n\ge 0}(A(m-1,n)-A(m,n))q^{n}.
\end{align}
Let ${\bf 1}_{event}$ be the indicator function, and for any two formal power series $f(q)$ and $g(q)$, the notation $f(q)\succeq g(q)$ means that all coefficients of $f(q)-g(q)$ are nonnegative. Our main results of this paper are as follows.

\begin{theorem}\label{main1}Let $m>c\ge 0$ be any integers. Then we have
$$D_{V^c}(m,q)\succeq -q{\bf 1}_{(c,m)=(0,1)}\;\;\text{and}\;\;D_{V_d^c}(m,q)\succeq -q{\bf 1}_{(c,m)=(0,1)}.$$
In particular, for all $(c, m, n)\neq (0,1,1)$ we have
$$V^{c}(m-1,n)\ge V^{c}(m,n)\;\;\text{and}\;\; V_d^{c}(m-1,n)\ge V_d^{c}(m,n).$$
\end{theorem}
\begin{remark}
By Proposition \ref{pro320} for $V^{c}(m,n)$, along with Proposition \ref{pro32} for $V_d^{c}(m,n)$, one can derive the precise conditions under which the strict inequalities $V^{c}(m-1,n) > V^{c}(m,n)$ and $V_d^{c}(m-1,n) > V_d^{c}(m,n)$ hold. We leave this to interested readers.
\end{remark}

\begin{theorem}\label{main2}Let $m\ge 2$. Then we have
$$q^{-m}D_V(m,q)\succeq q^3{\bf 1}_{m=2}+\frac{1}{(q^2)_\infty}(1-q^3{\bf 1}_{m=2}).$$
In particular,
$V(m-1,n)> V(m,n)$
for all $n\ge m$, except that $V(n-2,n)=V(n-1,n)$.
\end{theorem}

\begin{remark}
The inequality $V(0,n)\ge V(1,n)$ for any $n\ge 0$ remains open. Perhaps it can be solved by the method proposed in this paper (see the proof of this theorem), but the computation required is rather complicated, so we leave this problem to the interested reader.
\end{remark}

\begin{theorem}\label{main4}Let $m\ge 1$. Then we have
$$q^{-\binom{m}{2}}D_{V_d}(m,q)\succeq \frac{1+q^{1+2m}}{(q)_{m-1}(1-q^{1+m})}+(q^m)_1q^{2m+1}\sum_{ i\ge 0}\frac{ q^{i^2+(m+3)i}}{(q)_{i}(q)_{i+m}}.$$
In particular,
$V_d(m-1,n)>V_d(m,n)$
for all $n\ge \binom{m}{2}$, except that $V_d(0,1)=V_d(1,1)$.
\end{theorem}

The rest of the paper is organized as follows. In Section \ref{sec2}, we construct a difference system that will characterize the one variable rank generating functions. In Section \ref{sec3}, we study the monotonicity of $V^c(m,n)$ and $V(m,n)$ with respect to $m\ge 0$. In particular, we give the proof of Theorem \ref{main1} for $D_{V^c}(m,q)$ and Theorem \ref{main2}. In Section \ref{sec4}, we study the monotonicity of $V_d^c(m,n)$ and $V_d(m,n)$ with respect to $m\ge 0$. In particular, we give the proof of Theorem \ref{main1} for $D_{V_d^c}(m,q)$ and Theorem \ref{main4}. 
In Section \ref{sec5}, we close this paper with some remarks on the monotonicity of the Andrews-Garvan-Dyson cranks of integer partitions and another approach to the inequalities for $V(m,n)$ and $V_d(m,n)$.

\section{On a difference system and the underlying idea of the proof}\label{sec2}
In this section, we construct a \emph{difference system} that will characterize the one variable rank generating functions for $V^c(m,n)$, $V(m,n)$, $V_d^c(m,n)$ and $V_d(m,n)$ in the next two sections. Let $x,y$ be any real numbers and let $0<|q|<1$.
Define for any positive integer $a$ and nonnegative integer $r$ that
\begin{align*}
f_{a,r}&:=f_{a,r}(x,y)=\sum_{i\ge 0}\frac{q^{i^2+i(a-1)}x^i}{(q)_i(x)_{i}(xyq^{2i})_r},
\end{align*}
and
\begin{align*}
g_{a,r}&:=g_{a,r}(x,y)=\sum_{i\ge 0}\frac{q^{i^2+ia}x^i}{(q)_i(x)_{i+1}(xyq^{2i+1})_r}.
\end{align*}
Clearly, for any integer $r\ge 1$  we have
\begin{align}\label{eqo1}
f_{a,r}(x,y)-xy f_{a+2,r}(x,y)&=f_{a, r-1}(x, yq),\\ \label{eqo2}
g_{a,r}(x,y)-qxyg_{a+2,r}(x,y)&=g_{a,r-1}(x,qy).
\end{align}
For simplicity, we write
\begin{equation*}
f_{a}:=f_a(x):=f_{a,0}(x,0)\;\;\text{and}\;\; g_{a}:=g_a(x):=g_{a,0}(x,0).
\end{equation*}

The following difference system plays a crucial role in the proof of this paper.
\begin{proposition}\label{propm1}For any $\alpha, \beta$ we have
\begin{align*}
\alpha f_{a,r}+\beta g_{a,r}=&(\alpha+\beta)f_{a+1,r+1}+x(q^a\alpha+\beta) g_{a+1,r+1}\\
&-xy(q^r\alpha+\beta)f_{a+3,r+1}-x^2yq(q^{a} \alpha+q^{r}\beta) g_{a+3,r+1}.
\end{align*}
In particular,
$$
\alpha f_{a}+\beta g_{a}=(\alpha+\beta)f_{a+1}+x(q^a\alpha+\beta) g_{a+1}.$$
\end{proposition}

\begin{proof}
By the basic facts on $q$-shift factorial, we obtain
\begin{align*}
f_{a,r}&=\sum_{i\ge 0}\frac{q^{i^2+i(a-1)}x^i(1-q^i+q^i)}{(q)_i(x)_{i}(xyq^{2i})_r}\\
&=\sum_{i\ge 1}\frac{q^{i^2+i(a-1)}x^i}{(q)_{i-1}(x)_{i}(xyq^{2i})_r}+\sum_{i\ge 0}\frac{q^{i^2+ia}x^i}{(q)_i(x)_{i}(xyq^{2i})_r}\\
&=\sum_{i\ge 0}\frac{q^{i^2+i(a+1)+a}x^{i+1}(1-xyq^{2i+1})}{(q)_i(x)_{i+1}(xyq^{2i+1})_{r+1}}+\sum_{i\ge 0}\frac{q^{i^2+ia}x^i(1-xyq^{2i+r})}{(q)_i(x)_{i}(xyq^{2i})_{r+1}}\\
&=xq^a g_{a+1,r+1}-x^2yq^{a+1}g_{a+3,r+1}+f_{a+1,r+1}-xyq^r f_{a+3,r+1},
\end{align*}
and
\begin{align*}
g_{a,r}&=\sum_{i\ge 0}\frac{q^{i^2+ia}x^i(1-xq^{i}+xq^{i})}{(q)_i(x)_{i+1}(xyq^{2i+1})_r}\\
&=\sum_{i\ge 0}\frac{q^{i^2+ia}x^i}{(q)_i(x)_{i}(xyq^{2i+1})_r}+\sum_{i\ge 0}\frac{q^{i^2+i(a+1)}x^{i+1}}{(q)_i(x)_{i+1}(xyq^{2i+1})_r}\\
&=\sum_{i\ge 0}\frac{q^{i^2+ia}x^i(1-xyq^{2i})}{(q)_i(x)_{i}(xyq^{2i})_{r+1}}+\sum_{i\ge 0}\frac{q^{i^2+i(a+1)}x^{i+1}(1-xyq^{2i+r+1})}{(q)_i(x)_{i+1}(xyq^{2i+1})_{r+1}}\\
&=f_{a+1,r+1}-xyf_{a+3,r+1}+xg_{a+1, r+1}-x^2yq^{r+1}g_{a+3,r+1}.
\end{align*}
Thus the proof is complete by a straightforward way.
\end{proof}

\medskip

We now discuss the ideas underlying our proofs, which is partially inspired by Ji-Zang's proof of the unimodality of the Andrews-Garvan-Dyson's cranks \cite{MR4324846}. Using basic $q$-series identities, we obtain one-variable generating functions for the ranks that we considered. For example (see Proposition \ref{lem22}):
\begin{align*}
q^{-\frac{m(m-1)}{2}-mc}D_{V_d^c}(m,q)=\frac{1}{(q)_{m-1}}f_{1+2c}(q^m) -\frac{q^{m+c}}{(q)_{m-1}}g_{1+2c}(q^m),
\end{align*}
where $f_{1+2c}$ and $g_{1+2c}$ are two $q$-series has nonnegative coefficients defined as in the above, which essentially satisfy a difference system (see Proposition \ref{propm1}).
Note that this is a difference of two terms where each term has nonnegative coefficients. The right-hand side matches the form of the second identity in Proposition \ref{propm1}, i.e., a simple combination of $f_{1+2c}$ and $g_{1+2c}$. Therefore, applying Proposition \ref{propm1} yields:
\begin{align*}
q^{-\frac{m(m-1)}{2}-mc}D_{V_d^{c}}(m,q)=\frac{1-q^{m+c}}{(q)_{m-1}}f_{2c+2}(q^m)+\frac{q^{m+2c+1}(1-q^{m-c-1})}{(q)_{m-1}}g_{2c+2}(q^m).
\end{align*}
When $m\ge c+1$, the right-hand side becomes a sum of two terms, both possessing nonnegative coefficients. Then, the proof follows. In particular, we demonstrate that the two identities in Proposition \ref{propm1} (possibly requiring repeated  applications) are applicable to our rank functions in this paper, either analogously to or generalizing the case of $V_d^{c}(m,n)$ discussed above.

\section{Monotonicity of ranks in concave compositions}\label{sec3}
In this section, we use Proposition \ref{propm1} to study $D_{V^c}(m,q)$ and $D_{V}(m,q)$ for all $m>0$. In particular,
we establish the monotonicity of $V^c(m,n)$ and $V(m,n)$ with respect to $m\ge 0$. Recall that the $q$-binomial theorem states that 
\begin{equation}\label{eqmm0}
\frac{(\alpha t)_\infty}{(t)_\infty}=\sum_{n\ge 0}\frac{(\alpha)_n}{(q)_n}t^n,
\end{equation}
for any $\alpha\in \bc$ and $|t|<1$. Let $\alpha=0$ we obtain the well-known Euler's identity: 
\begin{equation}\label{eqmm1}
\frac{1}{(t)_\infty}=\sum_{n\ge 0}\frac{t^n}{(q)_n}.
\end{equation}
We first prove the following one-variable generating function for $V^c(m,n)$.
\begin{lemma}\label{lem31}
For $m\ge 1$, we have
\begin{align*}
\sum_{n\ge 0}V^{c-1}(m-1,n)q^n&=\frac{q^{mc-1}}{(q^{2c})_\infty}\sum_{i\ge 0}\frac{q^{i^2+i(m+2c-1)}}{(q)_i(q)_{i-1+m}}
\end{align*}
and
\begin{align*}
(q)_{m-1}(q^{2c})_\infty q^{1-mc}D_{V^{c-1}}(m,q)=f_{2c}(q^m)-q^cg_{2c}(q^m).
\end{align*}
\end{lemma}
\begin{proof}
Note that
\begin{align*}
\sum_{n\ge 0}\sum_{m\in\bz}V^{c-1}(m,n)q^n=\frac{q^{c-1}}{(zq^{c})_\infty(z^{-1}q^{c})_\infty}=\frac{q^{c-1}}{(q^{2c})_\infty(z^{-1}q^{c})_\infty}\frac{(q^{2c})_\infty}{(q^{c} z)_\infty}.
\end{align*}
Using $q$-binomial theorem \eqref{eqmm0} and Euler's identity \eqref{eqmm1}, we obtain
\begin{align*}
\sum_{n\ge 0}\sum_{m\in\bz}V^{c-1}(m,n)q^n
&=\frac{q^{c-1}}{(q^{2c})_\infty(z^{-1}q^{c})_\infty}\sum_{h\ge 0}\frac{(z^{-1}q^{c})_h}{(q)_h}(q^{c} z)^h\\
&=\frac{1}{(q^{2c})_\infty}\sum_{h\ge 0}\frac{z^hq^{c-1+hc}}{(q)_h(z^{-1}q^{c+h})_\infty}\\
&=\frac{1}{(q^{2c})_\infty}\sum_{i,h\ge 0}\frac{z^{h-i}q^{c-1+i(c+h)+hc}}{(q)_i(q)_h}.
\end{align*}
Thus, for $m\ge 1$ we have
\begin{align*}
\sum_{n\ge 0}V^{c-1}(m-1,n)q^n&=\frac{q^{mc-1}}{(q^{2c})_\infty}\sum_{i\ge 0}\frac{q^{i^2+i(m+2c-1)}}{(q)_i(q)_{i-1+m}}.
\end{align*}
Moreover, by using the definition of $f_a(x)$ and $g_a(x)$, we obtain
\begin{align*}
D_{V^{c-1}}(m,q)=&\frac{1}{(q^{2c})_\infty}\left(\sum_{i\ge 0}\frac{q^{i^2+i(m+2c-1)+mc-1}}{(q)_i(q)_{i-1+m}}-\sum_{i\ge 0}\frac{q^{i^2+i(m+2c)+(m+1)c-1}}{(q)_i(q)_{i+m}}\right)\\
=&\frac{q^{mc-1}}{(q)_{m-1}(q^{2c})_\infty}\left(\sum_{i\ge 0}\frac{q^{i^2+i(2c-1)}q^{mi}}{(q)_i(q^m)_{i}}-q^{c}\sum_{i\ge 0}\frac{q^{i^2+i(2c)}q^{mi}}{(q)_i(q^m)_{i+1}}\right)\\
=&\frac{q^{mc-1}}{(q)_{m-1}(q^{2c})_\infty}\left(f_{2c}(q^m)-q^cg_{2c}(q^m)\right),
\end{align*}
which completes the proof of the lemma.
\end{proof}
\subsection{Monotonicity of $V^c(m, n)$}
Using Lemma \ref{lem31} and the second identity in Proposition \ref{propm1}, we establish the following identities for $D_{V^{c-1}}(m,q)$, which will directly yields the monotonicity properties of \( V^c(m, n) \) in this paper.
\begin{proposition}\label{pro320}
For $m, c\ge 1$ we have
\begin{align*}
q^{1-mc}D_{V^{c-1}}(m,q)=&\frac{(q^c)_1 (q^{m+c})_1}{(q)_{m-1} (q^{2c})_\infty}f_{2c+2}(q^m)+q^{m+2c+1}\frac{ (q^c)_1 (q^{m-c-1})_1}{(q)_{m-1}(q^{2c})_\infty }g_{2c+2}(q^m).
\end{align*}
In particular, for $m>c\ge 1$ we have
\begin{align}\label{eqmcc1}
q^{1-mc}D_{V^{c-1}}(m,q)\succeq \frac{1}{(q^{2c})_\infty}\left(\frac{(q^c)_1(q^{m+c})_1}{(q)_{m-1}}+\frac{(q^c)_1(q^{m-c-1})_1q^{m+2c+1}}{(q)_m}\right).
\end{align}
For $m=c>1$ we have
\begin{align*}
\frac{D_{V^{c-1}}(c,q)}{q^{c^2-1}(q^c)_1}=&\frac{(q^{3c})_1(q)_1+q(q^{2c-1})_1}{(q^{2c})_\infty(q)_{c-1}}f_{2c+3}(q^c)+q^{3c+2}\frac{(q^{c-2})_1+q^{c-1}(q^{c+1})_1}{(q^{2c})_\infty(q)_{c-1}}g_{2c+3}(q^c).
\end{align*}
In particular, for $m=c>1$ we have
\begin{align}\label{eqmcc2}
q^{1-c^2}D_{V^{c-1}}(c,q)\succeq \frac{(q^c)_1}{(q)_{c-1}(q^{2c+2})_\infty}.
\end{align}
For $m=c=1$ we have
\begin{align*}
D_{V^{0}}(1,q)&=\frac{(q)_1^2(q^4)_1}{(q^2)_\infty}f_5(q)+\frac{q(q)_1^2}{(q^3)_\infty}f_4(q)+\frac{q^8(q)_1}{(q^2)_\infty}g_{4}(q^2).
\end{align*}
In particular, for $m=c=1$ we have
\begin{align}\label{eqmcc3}
D_{V^{0}}(1,q)\succeq \left((q)_1q^6+\frac{(q)_1q^{10}}{(q^2)_1}\right)\frac{1}{(q^2)_\infty}+\left((q)_1+\frac{q^6}{(q)_1}\right)\frac{1}{(q^4)_\infty}.
\end{align}
\end{proposition}
\begin{proof}
Recall the second identity in Proposition \ref{propm1} state that
\begin{align}\label{eq31}
\alpha f_{a}+\beta g_{a}=(\alpha+\beta)f_{a+1}+x(q^a\alpha+\beta) g_{a+1}.
\end{align}
With $x:=q^m$, $f_a=f_a(x)$ and $g_a=g_a(x)$, using of Lemma \ref{lem31} and \eqref{eq31} implies
\begin{align*}
(q)_{m-1}(q^{2c})_\infty q^{1-mc}D_{V^{c-1}}(m,q)&=f_{2c}-q^cg_{2c}\\
&=(1-q^c)f_{2c+1}+x(q^{2c}-q^{c})g_{2c+1}\\
&=(1-q^c)\left((1-xq^{c})f_{2c+2}+x(q^{2c+1}-xq^{c})g_{2c+2}\right),
\end{align*}
that is
\begin{align}\label{eqmcc0}
\frac{(q^{2c})_\infty(q)_{m-1} D_{V^{c-1}}(m,q)}{(q^c)_1 q^{mc-1}}=(q^{m+c})_1f_{2c+2}+q^{m+2c+1} (q^{m-c-1})_1g_{2c+2}.
\end{align}
Moreover, by the definition of $f_a(x)$ and $g_a(x)$, for all integers $c\ge 1$ and $m\ge c+1$ we have
\begin{align*}
q^{1-mc}D_{V^{c-1}}(m,q)
=&\frac{(q^c)_1(q^{m+c})_1}{(q^{2c})_\infty(q)_1(q)_{m-1}}\sum_{i\ge 1}\frac{q^{i^2+i(m+2c+1)}}{(q^2)_{i-1}(q^{m})_{i}}\\
&+\frac{(q^c)_1(q^{m-c-1})_1}{(q^{2c})_\infty(q)_1(q)_m}\sum_{i\ge 1}\frac{q^{i^2+i(m+2c+2)+m+2c-1}}{(q^2)_{i-1}(q^{m+1})_{i}}\\
&+\frac{1}{(q^{2c})_\infty}\left(\frac{(q^c)_1(q^{m+c})_1}{(q)_{m-1}}+\frac{(q^c)_1(q^{m-c-1})_1q^{m+2c+1}}{(q)_m}\right).
\end{align*}
Clearly, the coefficients in the power series expansions of the first and second terms on the right-hand side of the identity in the above formula are all non-negative. Therefore, we have
\begin{align*}
q^{1-mc}D_{V^{c-1}}(m,q)\succeq \frac{1}{(q^{2c})_\infty}\left(\frac{(q^c)_1(q^{m+c})_1}{(q)_{m-1}}+\frac{(q^c)_1(q^{m-c-1})_1q^{m+2c+1}}{(q)_m}\right),
\end{align*}
which completes the proof of \eqref{eqmcc1}. 

When $m=c>1$, the preceding expression \eqref{eqmcc0} can be simplified to:
\begin{align}\label{eqdvc}
\frac{(q)_{c-1}(q^{2c})_\infty D_{V^{c-1}}(c,q)}{(1-q^c)q^{c^2-1}}=&(1-q^{2c})f_{2c+2}-q^{3c}(1-q)g_{2c+2}.
\end{align}
In the above identity \eqref{eqdvc}, using \eqref{eq31} again implies
\begin{align*}
\frac{(q)_{c-1}(q^{2c})_\infty D_{V^{c-1}}(c,q)}{(1-q^c)q^{c^2-1}}=(1-q^{2c}-q^{3c}+q^{3c+1})f_{2c+3}+q^{3c+2}(1-q^{c-2}+q^{c-1}-q^{2c})g_{2c+3},
\end{align*}
that is
\begin{align*}
\frac{D_{V^{c-1}}(c,q)}{q^{c^2-1}(q^c)_1}=&\frac{(q^{3c})_1(q)_1+q(q^{2c-1})_1}{(q^{2c})_\infty(q)_{c-1}}f_{2c+3}+q^{3c+2}\frac{(q^{c-2})_1+q^{c-1}(q^{c+1})_1}{(q^{2c})_\infty(q)_{c-1}}g_{2c+3}.
\end{align*}
Moreover, by the definition of $f_a(x)$ and $g_a(x)$, for all integers $m=c>1$ we have
\begin{align*}
q^{1-c^2}D_{V^{c-1}}(c,q)=&\frac{(q^{3c})_1(q)_1+q(q^{2c-1})_1}{(q^{2c})_\infty(q)_{c-1}(q)_1}\sum_{i\ge 1}\frac{q^{i^2+i(3c+2)}}{(q^2)_{i-1}(q^{c+1})_{i-1}}\\
&+\frac{(q^{c-2})_1+q^{c-1}(q^{c+1})_1}{(q^{2c})_\infty(q)_1(q)_{c+1}}q^{3c+2}\sum_{i\ge 1}\frac{q^{i^2+i(3c+3)}}{(q^2)_{i-1}(q^{c+2})_{i-1}}\\
&+\left(\frac{(q^{3c})_1(q)_1(q^c)_1+q(q^{2c-1})_1(q^c)_1}{(q^{2c})_\infty(q)_{c-1}}+\frac{(q^{c-2})_1+q^{c-1}(q^{c+1})_1}{(q^{2c})_\infty(q)_{c-1}}q^{3c+2}\right).
\end{align*}
Clearly, the coefficients in the power series expansions for the first and second terms on the right-hand side of the identity in the above formula are all non-negative. Therefore, we have
\begin{align*}
q^{1-c^2}D_{V^{c-1}}(c,q)\succeq&\frac{(q^{3c})_1(q)_1(q^c)_1+q(q^{2c-1})_1(q^c)_1}{(q^{2c})_\infty(q)_{c-1}}+\frac{(q^{c-2})_1+q^{c-1}(q^{c+1})_1}{(q^{2c})_\infty(q)_{c-1}}q^{3c+2}\\
=& \frac{1-q^c-q^{2c}+q^{3c+1}+q^{3c+2}-q^{5c+2}}{(q)_{c-1}(q^{2c})_\infty}\\
=& \frac{(q^c)_1(q^{2c})_2+q^{2c+1}(q^{c-1})_1(1+q^{c+1})+q^{5c+1}(q)_1}{(q)_{c-1}(q^{2c})_\infty}\\
\succeq &\frac{(q^c)_1}{(q)_{c-1}(q^{2c+2})_\infty},
\end{align*} 
which completes the proof of \eqref{eqmcc2}.
  
When $m=c=1$, by \eqref{eqdvc} we have
\begin{align*}
D_{V^{0}}(1,q)&=\frac{(q)_1^2}{(q^2)_\infty}\left((1+q)f_4-q^3g_4\right)\\
&=\frac{(q)_1^2}{(q^2)_\infty}\left(f_4-q^4g_4+q(1-q^2)f_4+q^3(f_4-(1-q)g_4)\right).
\end{align*}
In the above identity, using \eqref{eq31} implies
\begin{align*}
f_4(q)-q^4g_4(q)&=(1-q^4)f_5(q)+q(q^4-q^4)g_5(q)\\
&=(1-q^4)f_5(q),
\end{align*}
and by the definition of $f_a(x)$ and $g_a(x)$, we obtain
\begin{align*}
f_4(q)-(q)_1g_4(q)&=\sum_{i\ge 0}\left(\frac{q^{i^2+4i}}{(q)_{i}^2}-\frac{(q)_1q^{i^2+5i}}{(q)_{i}(q)_{i+1}}\right)\\
&=\sum_{i\ge 0}\frac{1-q^i}{(q)_{i}(q)_{i+1}}q^{i^2+4i}=\sum_{i\ge 0}\frac{q^{i^2+6i+5}}{(q)_{i}(q)_{i+2}}=\frac{q^5}{(q)_1}g_{4}(q^2).
\end{align*}
Therefore,
\begin{align*}
\frac{(q^2)_\infty }{(q)_1^2} D_{V^{0}}(1,q)&=(q^4)_1f_5(q)+q(q^2)_1f_4(q)+\frac{q^8}{(q)_1}g_{4}(q^2).
\end{align*}
Moreover, by the definition of $f_a(x)$ and $g_a(x)$,  for $m=c=1$ we have
\begin{align*}
 D_{V^{0}}(1,q)=&\left( \frac{(q)_1q^6}{(q^2)_\infty}+\frac{(q)_1q^{10}}{(q^2)_\infty(q^2)_1}+\frac{(q)_1}{(q^4)_\infty}+\frac{q^6}{(q)_1(q^4)_\infty}\right)\\
&+\frac{q^6}{(q^2)_\infty}\sum_{i\ge 1}\frac{q^{i^2+6i}}{(q^2)_{i-1}(q^2)_{i}}+\frac{q^{10}}{(q^2)_\infty}\sum_{i\ge 1}\frac{q^{i^2+7i}}{(q^2)_{i-1}(q^2)_{i+1}}+\frac{1}{(q^4)_\infty(q)_2(q)_1}\sum_{i\ge 2}\frac{q^{i^2+5i}}{(q^2)_{i-2}^2}.
\end{align*}
Clearly, the coefficients in the power series expansions for the last three terms on the right-hand side of the identity in the above formula are all non-negative. Therefore, we have
\begin{align*}
 D_{V^{0}}(1,q)\succeq \frac{(q)_1q^6}{(q^2)_\infty}+\frac{(q)_1q^{10}}{(q^2)_\infty(q^2)_1}+\frac{(q)_1}{(q^4)_\infty}+\frac{q^6}{(q)_1(q^4)_\infty},
\end{align*}
which completes the proof of \eqref{eqmcc3}.  
\end{proof}

\begin{proof}[The proof of Theorem \ref{main1} for $D_{V^{c}}(m,q)$]
By the relation \eqref{eqmcc1} in Proposition \ref{pro320}, for $m>c> 1$ we have
\begin{align*}
q^{1-mc}D_{V^{c-1}}(m,q)&\succeq \frac{1}{(q^{m+c+1})_\infty} \frac{(q^{m+c})_1(q^c)_1}{(q^{2c})_{m-c+1}(q)_{m-1}}+\frac{q^{m+2c+1}}{(q^{2c})_\infty}\frac{(q^c)_1(q^{m-c-1})_1}{(q^2)_{m-1}(q)_1}\\
&\succeq \frac{1}{(q^{m+c+1})_\infty}\succeq 0;
\end{align*}
and by the relation \eqref{eqmcc2} in Proposition \ref{pro320}, for $m=c> 1$ we have
\begin{align*}
q^{1-c^2}D_{V^{c-1}}(c,q)\succeq \frac{(q^c)_1}{(q)_1}\frac{1}{(q^2)_{c-2}(q^{2c+2})_\infty}\succeq \frac{1}{(q^2)_{c-2}(q^{2c+2})_\infty}\succeq 0.
\end{align*}
Thus for all $m\ge c>1$, we have $D_{V^{c-1}}(m,q)\succeq 0$. It remains to prove the cases $m\ge c=1$. By the relation \eqref{eqmcc3} in Proposition \ref{pro320} and Euler's identity \eqref{eqmm1}, for $m=1$ we have
\begin{align*}
D_{V^{0}}(1,q)\succeq &\left((q)_1q^6+\frac{(q)_1q^{10}}{(q^2)_1}\right)\sum_{n\ge 0}\frac{q^{2n}}{(q)_n}+\left((q)_1+\frac{q^6}{(q)_1}\right)\sum_{n\ge 0}\frac{q^{4n}}{(q)_n}\\
\succeq &(q)_1q^6+\frac{(q)_1q^{10}}{(q^2)_1}+(q)_1+\frac{q^6}{(q)_1}\\
=&q^6-q^7+\sum_{n\ge 10}(-1)^{n}q^{n}+1-q+\sum_{n\ge 6}q^n\succeq -q.
\end{align*}
Similarly, by the relation \eqref{eqmcc1} in Proposition \ref{pro320}, for $m\ge 2$ we have
\begin{align*}
D_{V^{0}}(m,q)\succeq&\frac{(q^{m+1})_1q^{m-1}}{(q^{2})_\infty(q^2)_{m-2}}+\frac{(q^{m-2})_1q^{2m+2}}{(q^{2})_\infty(q^2)_{m-1}}.
\end{align*}
Thus if $m\neq 3$ we obtain $D_{V^{0}}(m,q)\succeq 0$. For $m=3$, by further use Euler's identity \eqref{eqmm1}, we obtain
\begin{align*}
q^{-2}D_{V^{0}}(m,q)\succeq&\frac{1+q^2}{(q^{2})_2}\sum_{n\ge 0}\frac{q^{4n}}{(q)_n}+\frac{(1-q)q^{6}}{(q^2)_{2}}\sum_{n\ge 0}\frac{q^{2n}}{(q)_n}\\
\succeq&\frac{1}{(q^{2})_2}+\frac{q^4}{(q^{2})_2(1-q)}+\frac{(1-q)q^{6}}{(q^2)_{2}}\succeq \frac{1}{(q^2)_2}\left(\sum_{n\ge 4}q^n-q^7\right)\succeq 0,
\end{align*}
that is $D_{V^{0}}(3,q)\succeq 0$. Therefore, for all $m\ge c\ge 1$ we have 
$$D_{V^{c-1}}(m,q)\succeq -q{\bf 1}_{(c,m)=(1,1)}.$$ 
In other words, for all integers $m>c\ge 0$ we have $D_{V^{c}}(m,q)\succeq -q{\bf 1}_{(c,m)=(0,1)}$,  which completes the proof of Theorem \ref{main1} for $D_{V^{c}}(m,q)$.
\end{proof}
\subsection{Monotonicity of $V(m, n)$}
We use the first identity in Lemma \ref{lem31} to establish the following identities for $V(m,n)$ and $D_{V}(m,q)$.
\begin{lemma}\label{lemm34}For $m\ge 0$ we have
\begin{align*}
\sum_{n\ge 0}V(m,n)q^n=\sum_{h\ge 0}\frac{q^{2h}}{(q)_h}\sum_{i\ge 0}\frac{q^{i^2+2i+1+(i+1)m}}{(1-q^{2h+2i+1+m})(q)_i(q)_{i+m}}.
\end{align*}
Moreover, for $m\ge 1$ we have
\begin{align*}
q^{-m}(q)_{m-1}D_V(m,q)=\sum_{h\ge 0}\frac{q^{2h}}{(q)_h}\left(f_{2,1}(q^m, q^{2h})-q g_{2,1}(q^m,q^{2h})\right).
\end{align*}
\end{lemma}
\begin{proof}By the first identity in Lemma \ref{lem31}, the definition of $V(m,n)$ and Euler's identity \eqref{eqmm1}, we obtain
\begin{align*}
\sum_{n\ge 0}V(m,n)q^n&=\sum_{c\ge 1}\frac{q^c}{(q^{2c})_\infty}\sum_{i\ge 0}\frac{q^{i^2+i(2c+m)+mc}}{(q)_i(q)_{i+m}}\\
&=\sum_{i\ge 0}\frac{q^{i^2+im}}{(q)_i(q)_{i+m}}\sum_{c\ge 1}q^{(m+1+2i)c}\sum_{h\ge 0}\frac{q^{2ch}}{(q)_h}\\
&=\sum_{h\ge 0}\frac{q^{2h}}{(q)_h}\sum_{i\ge 0}\frac{q^{i^2+2i+1+(i+1)m}}{(1-q^{2h+2i+1+m})(q)_i(q)_{i+m}}.
\end{align*}
This immediately implies
\begin{align*}
q^{-m}(q)_{m-1}D_V(m,q)=\sum_{h\ge 0}\frac{q^{2h}}{(q)_h}\left(\sum_{i\ge 0}\frac{q^{i^2+i}q^{mi}}{(q^mq^{2h+2i})_1(q)_i(q^m)_{i}}-\sum_{i\ge 0}\frac{q^{i^2+2i+1}q^{mi}}{(q^mq^{2h+2i+1})_1(q)_i(q^m)_{i+1}}\right).
\end{align*}
Recall that
$$f_{a,r}(x,y)=\sum_{i\ge 0}\frac{q^{i^2+i(a-1)}x^i}{(q)_i(x)_{i}(xyq^{2i})_r}\;\;\text{and}\;\; g_{a,r}(x,y)=\sum_{i\ge 0}\frac{q^{i^2+ia}x^i}{(q)_i(x)_{i+1}(xyq^{2i+1})_r},$$
we obtain the proof of the lemma.
\end{proof}
We now use the second identity in Proposition \ref{propm1}, that is
\begin{align}\label{eq33}
\alpha f_{a,r}+\beta g_{a,r}=&(\alpha+\beta)f_{a+1,r+1}+x(q^a\alpha+\beta) g_{a+1,r+1}\nonumber\\
&-xy(q^r\alpha+\beta)f_{a+3,r+1}-x^2yq(q^{a} \alpha+q^{r}\beta) g_{a+3,r+1},
\end{align}
to the second identity in Lemma \ref{lemm34}, which will establish the following for $D_{V}(m,q)$.
\begin{proposition}\label{propm2}
For $m\ge 2$ we have
$$q^{-m}D_V(m,q)=\sum_{h\ge 0}\frac{q^{2h}}{(q)_h(q^2)_{m-2}}\left(f_{3,2}(q^m, q^{2h})-q^{m+1}g_{3,2}(q^m, q^{2h})\right),
$$
where
\begin{align*}
&\frac{1}{(q^2)_{m-2}}\left(f_{3,2}(q^m, q^{2h})-q^{m+1}g_{3,2}(q^m, q^{2h})\right)\nonumber\\
=&q^{2}\frac{(q^{m-1})_1}{(q^2)_{m-2}}f_{4,2}(q^m,q^{1+2h})+q^{m+3}\frac{(q^{m-2})_1}{(q^2)_{m-2}}g_{4,2}(q^m,q^{1+2h})+\frac{(q^2)_1}{(q^2)_{m-2}(q^{m+2h})_3}\nonumber\\
&+\frac{1}{(q^2)_{m-2}}\sum_{i\ge 1}\frac{(q^2)_1 q^{i^2+(3+m)i}}{(q)_{i}(q^m)_{i}(q^{m+2h+2i})_2}\left(\frac{q^{m+2h+2i+2}}{ (q^{m+2h+2i+2})_1}+(q^{m-2})_1+q^{m-2}\frac{(q^{i-1+m+2h})_1}{(q^{m+2h+2i-1})_1}\right).
\end{align*}
In particular, for all $m\ge 2$ we have
\begin{align}\label{eqmmmm}
q^{-m}D_V(m,q) \succeq \sum_{h\ge 0}\frac{q^{2h}}{(q)_h(q^2)_{m-2}}\left(\frac{(q^{m-2})_1q^{m+3}}{(q^m)_{1}(q^{m+2h+2})_2}+\frac{(q^{m-1})_1}{(q^{m+2h})_2}+\frac{(q^2)_1 q^{m-1}}{(q^{m+2h})_3}\right).
\end{align}
\end{proposition}
\begin{proof}Setting $x:=q^m, y:=q^{2h}$,
$f_{2,1}:=f_{2,1}(q^m, q^{2h})$ and $g_{2,1}:= g_{2,1}(q^m,q^{2h})$.  By use of \eqref{eq33} we obtain
\begin{align*}
f_{2,1}-q g_{2,1}=&(1-q)f_{3,2}+x(q^2-q)g_{3,2}\\
=&(1-q)(f_{3,2}-xqg_{3,2}).
\end{align*}
Thus, by use of \eqref{eq33}  again, we obtain
\begin{align*}
f_{3,2}-xqg_{3,2}=&(1-xq)f_{4,3}+x(q^3-xq)g_{4,3}\\
&-xy(q^2-xq)f_{6,3}-x^2yq(q^3-q^2xq)g_{6,3}\\
=&q^{2}(q^{-1}x)_1\left(f_{4,3}-xyf_{6,3}\right)\\
&+xq^3(q^{-2}x)_1\left(g_{4,3}-qxyg_{6,3}\right)+(q^2)_1\left(f_{4,3}-q^2x^3yg_{6,3}\right).
\end{align*}
Using \eqref{eqo1} and \eqref{eqo2}, we see that
\begin{align*}
f_{4,3}(x,y)-xyf_{6,3}(x,y)&=f_{4,2}(x,yq),\\
g_{4,3}(x,y)-qxyg_{6,3}(x,y)&=g_{4,2}(x,yq).
\end{align*}
Note that
\begin{align*}
f_{4,3}-q^2x^3y g_{6,3}=&\sum_{i\ge 0}\frac{q^{i^2+3i}x^i}{(q)_i(x)_{i}(xyq^{2i})_3}-q^{2}x^3y\sum_{i\ge 0}\frac{q^{i^2+6i}x^{i} }{(q)_i(x)_{i+1}(xyq^{1+2i})_3}\\
=&\frac{1}{(xy)_3}+\sum_{i\ge 1}\frac{q^{i^2+3i}x^i}{(q)_{i}(x)_{i}(xyq^{2i})_3}-y\sum_{i\ge 1}\frac{q^{i^2+4i-3}x^{i+2} }{(q)_{i-1}(x)_{i}(xyq^{2i-1})_3}\\
=&\frac{1}{(xy)_3}+\sum_{i\ge 1}\frac{q^{i^2+3i}x^i}{(q)_{i}(x)_{i}(xyq^{2i})_2}\left(\frac{1}{ (xyq^{2i+2})_1}-\frac{(q^{i})_1 q^{i-3}x^2y}{(xyq^{2i-1})_1}\right),
\end{align*}
and
$$\frac{1}{ (xyq^{2i+2})_1}-\frac{(q^{i})_1 q^{i-3}x^2y}{(xyq^{2i-1})_1}=\frac{xyq^{2i+2}}{ (xyq^{2i+2})_1}+(q^{-2}x)_1+q^{-2}x\frac{(q^{i-1}xy)_1}{(xyq^{2i-1})_1}.$$
Thus, we obtain
\begin{align*}
f_{3,2}-xqg_{3,2}
=&q^{2}(q^{-1}x)_1f_{4,2}(x,yq)+xq^3(q^{-2}x)_1g_{4,2}(x,yq)\\
&+\frac{(q^2)_1}{(xy)_3}+\sum_{i\ge 1}\frac{(q^2)_1 q^{i^2+3i}x^i}{(q)_{i}(x)_{i}(xyq^{2i})_2}\left(\frac{xyq^{2i+2}}{ (xyq^{2i+2})_1}+(q^{-2}x)_1+q^{-2}x\frac{(q^{i-1}xy)_1}{(xyq^{2i-1})_1}\right).
\end{align*}
This immediately implies
\begin{align*}
&\frac{1}{(q^2)_{m-2}}\left(f_{3,2}(q^m, q^{2h})-q^{m+1}g_{3,2}(q^m, q^{2h})\right)\nonumber\\
=&q^{2}\frac{(q^{m-1})_1}{(q^2)_{m-2}}f_{4,2}(q^m,q^{1+2h})+q^{m+3}\frac{(q^{m-2})_1}{(q^2)_{m-2}}g_{4,2}(q^m,q^{1+2h})+\frac{(q^2)_1}{(q^2)_{m-2}(q^{m+2h})_3}\nonumber\\
&+\sum_{i\ge 1}\frac{(q^2)_1 q^{i^2+(3+m)i}}{(q^2)_{m-2} (q)_{i}(q^m)_{i}(q^{m+2h+2i})_2}\left(\frac{q^{m+2h+2i+2}}{ (q^{m+2h+2i+2})_1}+(q^{m-2})_1+q^{m-2}\frac{(q^{i-1+m+2h})_1}{(q^{m+2h+2i-1})_1}\right).
\end{align*}
Clearly, for any $m\ge 2$ and $h\ge 0$, the coefficients in the summation on the right-hand side of the identity in the above formula are all non-negative. Therefore,  by the definition of $f_{a,r}$ and $g_{a,r}$, we obtain
\begin{align*}
&\frac{f_{3,2}(q^m, q^{2h})-q^{m+1}g_{3,2}(q^m, q^{2h})}{(q^2)_{m-2}}\\
\succeq & q^{2}\frac{(q^{m-1})_1}{(q^2)_{m-2}}f_{4,2}(q^m,q^{1+2h})+q^{m+3}\frac{(q^{m-2})_1}{(q^2)_{m-2}}g_{4,2}(q^m,q^{1+2h})+\frac{(q^2)_1}{(q^2)_{m-2}(q^{m+2h})_3} \\ = &\frac{(q^{m-1})_1q^{2}}{(q^2)_{m-2}(q^{m+2h+1})_2}+\frac{(q^{m-2})_1q^{m+3}}{(q^2)_{m-1}(q^{m+2h+2})_2}+\frac{(q^2)_1}{(q^2)_{m-2}(q^{m+2h})_3}\\
&+q^{2}\frac{(q^{m-1})_1}{(q^2)_{m-2}}\sum_{i\ge 1}\frac{q^{i^2+(3+m)i}}{(q)_i(q^m)_{i}(q^{m+1+2h+2i})_2}+q^{m+3}\frac{(q^{m-2})_1}{(q^2)_{m-2}}\sum_{i\ge 1}\frac{q^{i^2+(4+m)i}}{(q)_i(q^m)_{i+1}(q^{m+2+2h+2i})_2}\\
\succeq & \frac{1}{(q^2)_{m-2}}\left(\frac{(q^{m-1})_1q^{2}}{(q^{m+2h+1})_2}+\frac{(q^{m-2})_1q^{m+3}}{(q^m)_{1}(q^{m+2h+2})_2}+\frac{(q^2)_1}{(q^{m+2h})_3}\right).
\end{align*}
Therefore, for $m\ge 2$ we have
\begin{align*}
q^{-m}D_V(m,q)&=\sum_{h\ge 0}\frac{q^{2h}}{(q)_h}\frac{f_{3,2}(q^m, q^{2h})-q^{m+1}g_{3,2}(q^m, q^{2h})}{(q^2)_{m-2}}\\
&\succeq \sum_{h\ge 0}\frac{q^{2h}}{(q)_h(q^2)_{m-2}}\left(\frac{(q^{m-1})_1q^{2}}{(q^{m+2h+1})_2}+\frac{(q^{m-2})_1q^{m+3}}{(q^m)_{1}(q^{m+2h+2})_2}+\frac{(q^2)_1}{(q^{m+2h})_3}\right).
\end{align*}
Note that 
$$\frac{(q^{m-1})_1q^{2}}{(q^{m+2h+1})_2}+\frac{(q^2)_1}{(q^{m+2h})_3}=\frac{(q^{m-1})_1}{(q^{m+2h})_2}+\frac{(q^2)_1 q^{m-1}}{(q^{m+2h})_3},$$
we complete the proof of the proposition.
\end{proof}
We now use the relation \eqref{eqmmmm} in Proposition \ref{propm2} to give the proof of Theorem \ref{main2}.
\begin{proof}[The proof of Theorem \ref{main2}]
For $m\ge 2$, by use of \eqref{eqmmmm} in Proposition \ref{propm2} we obtain
\begin{align*}
q^{-m}D_V(m,q)\succeq &\sum_{h\ge 0}\frac{q^{2h}}{(q)_h}\left(\frac{(q^{m-2})_1q^{m+3}}{(q^2)_{m-1}(q^{m+2h+2})_2}+\frac{(q^{m-1})_1}{(q^2)_{m-2}(q^{m+2h})_2}+\frac{(q^2)_1 q^{m-1}}{(q^2)_{m-2}(q^{m+2h})_3}\right).
\end{align*}
Thus for $m=2$, we obtain 
\begin{align*}
q^{-2}D_V(2,q)\succeq &\sum_{h\ge 0}\frac{q^{2h}}{(q)_h}\left(\frac{(q)_1}{(q^{2+2h})_2}+\frac{(q^2)_1 q}{(q^{1+2h})_3}\right)\\
\succeq &\frac{(q)_1}{(q^{2})_2}+\frac{(q^2)_1 q}{(q^{2})_3}+\frac{q^{2}}{(q)_1}\left(\frac{(q)_1}{(q^{4})_2}+\frac{(q^2)_1 q}{(q^{3})_3}\right)+\sum_{h\ge 2}\frac{q^{2h}}{(q)_h}\left((q)_1+(q^2)_1 q\right)\\
=&\frac{1}{(q^2)_1}+\frac{q^5}{(q^3)_2}+\frac{q^2}{(q^4)_2}+\frac{(1+q)q^3}{(q^3)_3}+(1-q^3)\left(\frac{1}{(q^2)_\infty}-1-\frac{q^2}{1-q}\right)\\
= & \frac{q^6}{(q^2)_1}+\frac{q^5}{(q^3)_2}+\frac{q^2}{(q^4)_2}+\frac{(1+q)q^3}{(q^3)_3}+\frac{1-q^3}{(q^2)_\infty}\\
\succeq & q^3+(1-q^3)\frac{1}{(q^2)_\infty}.
\end{align*}
For $m\ge 3$, we obtain 
\begin{align*}
q^{-m}D_V(m,q)\succeq &\frac{(q^{m-2})_1q^{m+3}}{(q^2)_{m-1}(q^{m+2})_2}+\sum_{h\ge 0}\frac{q^{2h}}{(q)_h}\left(\frac{1}{(q^2)_{m-3}(q^{m+2h})_2}+\frac{ q^{m-1}}{(q^3)_{m-3}(q^{m+2h})_3}\right)\\
\succeq &\frac{q^{m-2}-q^{2m+1}}{(q^2)_{m-1}(q^{m+2})_2}+(1+q^{m-1})\sum_{h\ge 0}\frac{q^{2h}}{(q)_h}\\
=&\frac{q^{m-2}}{(q^2)_{m-1}(q^{m+2})_2}+\frac{1}{(q^2)_\infty}+\frac{q^{m-1}}{(q^2)_\infty}-\frac{q^{2m+1}}{(q^2)_{m-1}(q^{m+2})_2}.
\end{align*}
Therefore, by noting that
\begin{align*}
\frac{q^{m-1}}{(q^2)_\infty}-\frac{q^{2m+1}}{(q^2)_{m-1}(q^{m+2})_2}\succeq &\frac{q^{m-1}}{(q^2)_{m-1}(q^{m+2})_2}-\frac{q^{2m+1}}{(q^2)_{m-1}(q^{m+2})_2}\\
=&\frac{q^{m-1}}{(q^2)_{m-1}(q^{m+3})_1}\succeq 0,
\end{align*}
we obtain 
$$q^{-m}D_V(m,q)\succeq q^3{\bf 1}_{m=2}+(1-q^3{\bf 1}_{m=2})\frac{1}{(q^2)_\infty}\succeq 1+\sum_{n\ge 2}q^2,$$
holds for any integer $m\ge 2$. This immediately yields $V(m-1,n)> V(m,n)$
for all $1< m\le n$, except that $V(n-2,n)\ge V(n-1,n)$. On the other hand, by the first identity in Lemma \ref{lem31}, we have
\begin{align*}
\sum_{n\ge 0}V(m,n)q^n=&\sum_{h\ge 0}\frac{q^{2h}}{(q)_h}\sum_{i\ge 0}\frac{q^{i^2+2i+1+(i+1)m}}{(1-q^{2h+2i+1+m})(q)_i(q)_{i+m}}\\
=&\sum_{i\ge 0}\frac{q^{i^2+2i+1+(i+1)m}}{(1-q^{2i+1+m})(q)_i(q)_{i+m}}+O(q^{3+m})\\
&=\frac{q^{1+m}}{(1-q^{1+m})(q)_{m}}+O(q^{3+m})=q^{1+m}+q^{2+m}+O(q^{m+3}).
\end{align*}
Thus for any $m\ge 0$, we have $V(m,m+1)=V(m,m+2)=1$, that is 
$$V(n-2,n)-V(n-1,n)=1-1=0,$$
holds for any $n\ge 2$, which completes the proof. 
\end{proof}
\section{Monotonicity of ranks in strongly concave compositions}\label{sec4}
In this section, we use Proposition \ref{propm1} to study the monotonicity of $V_d^c(m,n)$ and $V_d(m,n)$ with respect to $m\ge 0$. We begin with an one-variable generating function for $V_d^c(m,n)$.
\begin{proposition}\label{lem22}
For any $m\ge 0$,
\begin{align*}
\sum_{n\ge 0}V_d^c(m,n)q^n=q^{\frac{m(m+1)}{2}}\sum_{ i\ge 0}\frac{q^{i^2+(m+2c+1)i+ (1+m)c}}{(q)_{i}(q)_{i+m}}.
\end{align*}
Moreover, for $m\ge 1$ we have
\begin{align*}
(q)_{m-1}q^{-\binom{m}{2}-mc}D_{V_d^c}(m,q)=f_{1+2c}(q^m) -q^{m+c}g_{1+2c}(q^m).
\end{align*}
\end{proposition}

\begin{proof}Using the well--known Euler identity, that is
$$(-x;q)_\infty=\sum_{\ell\ge 0}\frac{q^{\frac{\ell(\ell-1)}{2}}x^\ell}{(q)_\ell},$$
we have
\begin{align*}
\sum_{n\ge 0}\sum_{m\in\bz}V_d^c(m,n)z^{m}q^{n-c}&=(-zq^{c+1};q)_\infty (-z^{-1}q^{c+1};q)_\infty\\
&=\sum_{\ell_1, \ell_2\ge 0}\frac{q^{\frac{\ell_1(\ell_1+1)}{2}+\frac{\ell_2(\ell_2+1)}{2}+c(\ell_1+\ell_2)}z^{\ell_1-\ell_2}}{(q)_{\ell_1}(q)_{\ell_2}}.
\end{align*}
This immediately implies
\begin{align*}
\sum_{n\ge 0}V_d^c(m,n)q^n=\sum_{\substack{\ell_1, \ell_2\ge 0\\ \ell_1-\ell_2=m}}\frac{q^{(\ell_1+\ell_2+1)c+\frac{\ell_1(\ell_1+1)}{2}+\frac{\ell_2(\ell_2+1)}{2}}}{(q)_{\ell_1}(q)_{\ell_2}}=q^{\frac{m(m+1)}{2}}\sum_{ i\ge 0}\frac{q^{i^2+(m+2c+1)i+ (1+m)c}}{(q)_{i}(q)_{i+m}},
\end{align*}
for all $m\ge 0$. Thus,
\begin{align*}
D_{V_d^c}(m,q)=&q^{\frac{m(m-1)}{2}}\sum_{ i\ge 0}\frac{q^{i^2+(m+2c)i+ mc}}{(q)_{i}(q)_{i-1+m}}-q^{\frac{m(m+1)}{2}}\sum_{ i\ge 0}\frac{q^{i^2+(m+2c+1)i+ (1+m)c}}{(q)_{i}(q)_{i+m}}\\
=&\frac{q^{\frac{m(m-1)}{2}+mc}}{(q)_{m-1}}\left(\sum_{ i\ge 0}\frac{q^{i^2+2ci}q^{mi}}{(q)_{i}(q^m)_{i}}-q^{m+c}\sum_{ i\ge 0}\frac{q^{i^2+(2c+1)i}q^{mi}}{(q)_{i}(q^m)_{i+1}}\right)\\
=&\frac{q^{\frac{m(m-1)}{2}+mc}}{(q)_{m-1}}\left(f_{1+2c}(q^m) -q^{m+c}g_{1+2c}(q^m)\right).
\end{align*}
This complete the proof.
\end{proof}
\subsection{Monotonicity of $V_d^c(m, n)$}
Using Lemma \ref{lem31} and the first identity in Proposition \ref{propm1}, we establish the following identities for $D_{V_d^{c-1}}(m,q)$, which will directly yield the monotonicity properties of \( V_d^c(m, n) \) in this paper.
\begin{proposition}\label{pro32}
For $c, m\ge 1$ we have
\begin{align*}
q^{-\binom{m}{2}-m(c-1)}D_{{ V_d^{c-1}}}(m, q)=\frac{(q^{m+c-1})_1}{(q)_{m-1}}f_{2c}(q^m)+q^{m+2c-1}\frac{(q^{m-c})_1}{(q)_{m-1}}g_{2c}(q^m).
\end{align*}
In particular, for $m> c\ge 0$ we have
\begin{align*}
q^{-\binom{m}{2}-mc}D_{{ V_d^{c}}}(m, q)
\succeq &\frac{(q^{m+c})_1}{(q)_{m-1}}+q^{m+2c+1}\frac{(q^{m-c-1})_1}{(q)_{m}}.
\end{align*}
\end{proposition}
\begin{proof}
Recall the second identity in Proposition \ref{propm1} that
$$
\alpha f_{a}+\beta g_{a}=(\alpha+\beta)f_{a+1}+x(q^a\alpha+\beta) g_{a+1},$$
and the using of Lemma \ref{lem22} implies
\begin{align*}
(q)_{m-1}q^{-\frac{m(m-1)}{2}-m(c-1)} D_{V_d^{c-1}}(m,q)&=f_{2c-1}(q^m) -q^{m+c-1}g_{2c-1}(q^m)\\
&=(1-q^{m+c-1})f_{2c}(q^m)+q^{m}(q^{2c-1}-q^{m+c-1})g_{2c}(q^m).
\end{align*}
Therefore, by using the definition of $f_{a}$ and $g_{a}$, we obtain
\begin{align*}
q^{-\binom{m}{2}-m(c-1)}D_{{ V_d^{c-1}}}(m, q)=&\frac{(q^{m+c-1})_1}{(q)_{m-1}}\sum_{i\ge 0}\frac{q^{i^2+(2c-1)i+mi}}{(q)_i(q^m)_{i}}+q^{m+2c-1}\frac{(q^{m-c})_1}{(q)_{m}}\sum_{i\ge 0}\frac{q^{i^2+2ci+mi}}{(q)_i(q^{m+1})_{i}}.
\end{align*}
Moreover, for $m\ge c\ge 1$ we have
\begin{align*}
q^{-\binom{m}{2}-m(c-1)}D_{{ V_d^{c-1}}}(m, q)
\succeq &\frac{(q^{m+c-1})_1}{(q)_{m-1}}+q^{m+2c-1}\frac{(q^{m-c})_1}{(q)_{m}}.
\end{align*}
Letting $c\mapsto c+1$ we complete the proof of the proposition.
\end{proof}

\begin{proof}[The proof of Theorem \ref{main1} for $D_{V_d^{c}}(m,q)$]
For all $c\ge 0$ and $m\ge \max(2, c+1)\ge 2$, by Proposition \ref{pro32}, it is clear that
\begin{align*}
q^{-\binom{m}{2}-mc}D_{{ V_d^{c}}}(m, q)
\succeq &\frac{(q^{m+c})_1}{(q)_1(q^2)_{m-2}}+q^{m+2c+1}\frac{(q^{m-c-1})_1}{(q)_{m}}\succeq \frac{1}{(q^2)_{m-2}}\succeq 0.
\end{align*}
It remains to prove the case $(c, m)=(0,1)$. By Proposition \ref{pro32},
\begin{align*}
D_{{ V_d^{0}}}(1, q)\succeq 1-q\succeq -q.
\end{align*}
Thus $D_{V_d^{c}}(m,q)\succeq -q{\bf 1}_{(c,m)=(0,1)}$, which complete the proof of Theorem \ref{main1} for $D_{V_d^{c}}(m,q)$.
\end{proof}

\subsection{Monotonicity of $V_d(m, n)$}We use the first identity in Lemma \ref{pro32} to establish the following identities for $V(m,n)$ and $D_{V}(m,q)$.
\begin{proposition}\label{pth43}
For all $m\ge 0$ we have
\begin{align*}
\sum_{n\ge 0}V_d(m,n)q^n=q^{\frac{m(m+1)}{2}}\sum_{ i\ge 0}\frac{q^{i^2+(m+1)i}}{(q)_{i}(q)_{i+m}(1-q^{2i+1+m})}.
\end{align*}
For $m\ge 1$ and with $x=q^m$, we have
\begin{align*}
&q^{-{m(m-1)}/{2}}D_{V_d}(m,q)\\
=&\frac{1}{(q)_{m-1}(xq)_1}+\frac{1}{(q)_{m-1}}\sum_{ i\ge 1}\frac{q^{i^2+i-1}x^{i+1}(1+xq^{3i+1})}{(q)_{i-1}(xq)_{i-1}(xq^{2i-1})_3}\\
&+\frac{(q^{-1}x)_1}{(q)_{m-1}}\left(\sum_{ i\ge 1}\frac{q^{i^2+i}x^{i}}{(q)_{i}(xq)_{i-1}(xq^{2i-1})_3}+\sum_{ i\ge 1}\frac{q^{i^2+2i+1}x^{i+1}}{(q)_{i-1}(x)_{i+1}(xq^{2i})_3}+\sum_{ i\ge 0}\frac{q^{i^2+5i+3}x^{i+2}}{(q)_{i}(x)_{i}(xq^{2i})_3}\right).
\end{align*}
In particular,
\begin{align}\label{eqd1}
D_{V_d}(1,q)=\frac{1}{1-q^2}+(1-q)q^3\sum_{ i\ge 0}\frac{q^{i^2+4i}(1+q^{3i+5})}{(q)_{i}(q)_{i+1}(q^{2+2i})_3}.
\end{align}
\end{proposition}
\begin{proof}
By Lemma \ref{lem22}, we obtain
\begin{align*}
\sum_{n\ge 0}V_d(m,n)q^n&=\sum_{c\ge 0}\sum_{ i\ge 0}\frac{q^{i^2+(m+2c+1)i+ (1+m)c+\frac{m(m+1)}{2}}}{(q)_{i}(q)_{i+m}}=\sum_{ i\ge 0}\frac{q^{i^2+(m+1)i+\frac{m(m+1)}{2}}}{(q)_{i}(q)_{i+m}(1-q^{2i+1+m})}.
\end{align*}
Setting $x:=q^m$, $f_{1,1}:=f_{1,1}(q^m, 1)$ and $g_{1,1}:= g_{1,1}(q^m,1)$.  We find that
\begin{align*}
\frac{(q)_{m-1}D_{V_d}(m,q)}{q^{{m(m-1)}/{2}}}&=\sum_{ i\ge 0}\frac{q^{i^2+mi}}{(q)_{i}(q^m)_{i}(q^{m+2i})_1}-\sum_{ i\ge 0}\frac{q^{i^2+(m+1)i+m}}{(q)_{i}(q^m)_{i+1}(q^{2i+1+m})_1}\\
&=f_{1,1}-x g_{1,1}.
\end{align*}
Recall the first identity in Proposition \ref{propm1} states that
\begin{align*}
\alpha f_{a,r}+\beta g_{a,r}=&(\alpha+\beta)f_{a+1,r+1}+x(q^a\alpha+\beta) g_{a+1,r+1}\\
&-xy(q^r\alpha+\beta)f_{a+3,r+1}-x^2yq(q^{a} \alpha+q^{r}\beta) g_{a+3,r+1},
\end{align*}
where
$$f_{a,r}(x,y)=\sum_{i\ge 0}\frac{q^{i^2+i(a-1)}x^i}{(q)_i(x)_{i}(xyq^{2i})_r}\;\;\text{and}\;\; g_{a,r}(x,y)=\sum_{i\ge 0}\frac{q^{i^2+ia}x^i}{(q)_i(x)_{i+1}(xyq^{2i+1})_r}.$$
Therefore, by using the above difference relations we obtain  
\begin{align}\label{eq111}
q^{-\binom{m}{2}}(q)_{m-1}D_{V_d}(m,q)=&f_{1,1}-x g_{1,1}\nonumber\\
=&(1-x)f_{2,2}+x(q-x) g_{2,2}-x(q-x)f_{4,2}-x^2q(q-qx) g_{4,2}\nonumber\\
=&(1-x)\left(f_{2,2}-x^2q^2g_{4,2}\right)+xq(1-q^{-1}x)\left(g_{2,2}-f_{4,2}\right).
\end{align}
Using the definition of $f_{a,r}$ and $g_{a,r}$, we find that
\begin{align*}
f_{2,2}-x^2q^2g_{4,2}
=&\sum_{ i\ge 0}\frac{q^{i^2+i}x^i}{(q)_{i}(x)_{i}(xq^{2i})_2}-q^{2}x^2\sum_{ i\ge 0}\frac{q^{i^2+4i}x^i}{(q)_{i}(x)_{i+1}(xq^{2i+1})_2}\\
=&\frac{1}{(x)_2}+\sum_{i\ge 1}\frac{q^{i^2+i}x^i(1-xq^{2i-1})}{(q)_{i}(x)_{i}(xq^{2i-1})_3}-\sum_{ i\ge 1}\frac{q^{i^2+2i-1}x^{i+1}(1-q^i)(1-xq^{2i+1})}{(q)_{i}(x)_{i}(xq^{2i-1})_3}\\
=&\frac{1}{(x)_2}+\sum_{ i\ge 1}\frac{q^{i^2+i-1}x^{i}(q-x+x-q^{i}x+x^2q^{3i+1}-x^2q^{4i+1})}{(q)_{i}(x)_{i}(xq^{2i-1})_3}\\
=&\frac{1}{(x)_2}+\sum_{ i\ge 1}\frac{q^{i^2+i}x^{i}(1-q^{-1}x)}{(q)_{i}(x)_{i}(xq^{2i-1})_3}+\sum_{ i\ge 1}\frac{q^{i^2+i-1}x^{i+1}(1+xq^{3i+1})}{(q)_{i-1}(x)_{i}(xq^{2i-1})_3},
\end{align*}
and
\begin{align*}
g_{2,2}-f_{4,2}=&\sum_{ i\ge 0}\frac{q^{i^2+2i}x^i}{(q)_{i}(x)_{i+1}(xq^{2i+1})_2}-\sum_{ i\ge 0}\frac{q^{i^2+3i}x^i}{(q)_{i}(x)_{i}(xq^{2i})_2}\\
=&\sum_{ i\ge 0}\frac{q^{i^2+2i}x^i(1-xq^{2i})-q^{i^2+3i}x^i(1-xq^i)(1-xq^{2i+2})}{(q)_{i}(x)_{i+1}(xq^{2i})_3}\\
=&\sum_{ i\ge 0}\frac{q^{i^2+2i}x^i(1-q^i)+q^{i^2+5i+2}x^{i+1}(1-xq^{i})}{(q)_{i}(x)_{i+1}(xq^{2i})_3}\\
=&\sum_{ i\ge 1}\frac{q^{i^2+2i}x^i}{(q)_{i-1}(x)_{i+1}(xq^{2i})_3}+\sum_{ i\ge 0}\frac{q^{i^2+5i+2}x^{i+1}}{(q)_{i}(x)_{i}(xq^{2i})_3}.
\end{align*}
Therefore, by inserting the above into \eqref{eq111}, we obtain
\begin{align*}
&q^{-{m(m-1)}/{2}}D_{V_d}(m,q)\\
=&\frac{1}{(q)_{m-1}(xq)_1}+\frac{1}{(q)_{m-1}}\sum_{ i\ge 1}\frac{q^{i^2+i-1}x^{i+1}(1+xq^{3i+1})}{(q)_{i-1}(xq)_{i-1}(xq^{2i-1})_3}\\
&+\frac{(q^{-1}x)_1}{(q)_{m-1}}\left(\sum_{ i\ge 1}\frac{q^{i^2+i}x^{i}}{(q)_{i}(xq)_{i-1}(xq^{2i-1})_3}+\sum_{ i\ge 1}\frac{q^{i^2+2i+1}x^{i+1}}{(q)_{i-1}(x)_{i+1}(xq^{2i})_3}+\sum_{ i\ge 0}\frac{q^{i^2+5i+3}x^{i+2}}{(q)_{i}(x)_{i}(xq^{2i})_3}\right),
\end{align*}
which completes the proofs.
\end{proof}
Based on Proposition \ref{pth43}, we give the proof of Theorem \ref{main4} as follows.
\begin{proof}[The proof of Theorem \ref{main4}]By the second identity in Proposition \ref{pth43}, we have
\begin{align*}
q^{-\binom{m}{2}}D_{V_d}(m,q)\succeq &\frac{1}{(q)_{m-1}(q^{m+1})_1}+\frac{1}{(q)_{m-1}}\sum_{ i\ge 0}\frac{ q^{i^2+(m+3)i+2m+1}(1+q^{3i+m+4})}{(q)_{i}(q^{m+1})_{i}(q^{m+2i+1})_3}\\
\succeq &\frac{1}{(q)_{m-1}(q^{m+1})_1}+\frac{q^{2m+1}(1+q^{m+4})}{(q)_{m-1}(q^{m+1})_3}+\frac{(q^m)_1}{(q)_{m}}\sum_{ i\ge 0}\frac{ q^{i^2+(m+3)i+2m+1}}{(q)_{i}(q^{m+1})_{i}}\\
\succeq &\frac{1+q^{2m+1}}{(q)_{m-1}(q^{m+1})_1}+(1-q^m)q^{2m+1}\sum_{ i\ge 0}\frac{ q^{i^2+(m+3)i}}{(q)_{i}(q)_{i+m}}.
\end{align*}
In particular,
\begin{align*}
D_{V_d}(1,q)\succeq &\frac{1+q^{3}}{1-q^{2}}+(1-q)q^{3}\sum_{ i\ge 0}\frac{ q^{i^2+4i}}{(q)_{i}(q)_{i+1}}\\
=&\frac{1+q-q(1-q^2)}{1-q^2}+\sum_{ i\ge 0}\frac{ q^{i^2+4i+3}}{(q)_{i}(q^2)_{i}}\succeq \frac{1}{1-q}-q,
\end{align*}
and for $m\ge 2$,
\begin{align*}
q^{-\binom{m}{2}}D_{V_d}(m,q)\succeq \frac{1+q^{2m+1}}{(q)_1 (q^2)_{m-2}(q^{m+1})_1}+\sum_{ i\ge 0}\frac{ q^{i^2+(m+3)i+2m+1}}{(q)_{i}(q)_{m-1}(q^{m+1})_{i-1}}\succeq \frac{1}{1-q}.
\end{align*}
By the definition of $D_{V_d}(m,q)$, the above yields  $V_d(m-1,n)>V_d(m,n)$
for all integers $m\ge 1$ and $n\ge \binom{m}{2}$, except that $V_d(0,1)\ge V_d(1,1)$. By the identity \eqref{eqd1} in Proposition \ref{pth43}, we see that
$V_d(0,1)=V_d(1,1)$, which completes the proof.  
\end{proof}

\section{Final remarks}\label{sec5}
In this section, we give a discussion of how our results connect to the monotonicity of cranks of integer partitions, and present some remarks on the proofs of the monotonicity of ranks for (strongly) concave compositions.
\subsection{Monotonicity of cranks of integer partitions} Our method also can applications to study the monotonicity of cranks of integer partitions. The Andrews-Garvan-Dyson crank statistic for integer partitions was introduced and investigated by Dyson \cite{MR3077150}, Andrews and Garvan \cite{MR929094, MR920146}, to give a explaination for famous partition congruences with modulus $5$, $7$ and $11$ of Ramanujan \cite{MR2280868}. As a precise definition of crank for integer partitions are not necessary for the following content, we do not give it here. Let $M(m, n)$ (with a slight modification in the case that $n = 1$, where the values are instead $M(\pm 1, 1) = 1, M(0, 1) = -1$) denote the number of partitions of $n$ with crank $m$. It is well--known that
\begin{align*}
\mathcal{C}(z,q):=\sum_{m\in\bz}\sum_{n\ge 0}M(m,n)z^mq^n=\frac{(q)_\infty}{(zq)_\infty (z^{-1}q)_\infty}.
\end{align*} 
In \cite{MR4324846}, Ji and Zang proved that for all $m \geq 1$ and $n \geq \max(44, m+1)$, one has the following inequality for the crank function $M(m,n)$:
\begin{equation}\label{JZie}
M(m-1,n) \geq M(m, n).
\end{equation}

By notice that $\mathcal{C}(z,q)=(q)_\infty \mathcal{V}^0(z,q)$, we have $D_{M}(m,q)=(q)_\infty D_{V^0}(m,q)$. Thus, by using Proposition \ref{pro320} we immediately obtain the following proposition.   
\begin{proposition}\label{pro5}
We have
\begin{align*}
D_{M}(1,q)
&=(q)_1(q)_3\sum_{i\ge 0}\frac{q^{i^2+5i}}{(q)_i(q)_{i+1}}+(q)_1^3\sum_{i\ge 0}\frac{q^{i^2+6i+6}}{(q)_{i}(q)_{i+1}}+(q)_1^3\sum_{i\ge 0}\frac{q^{i^2+7i+10}}{(q)_{i}(q)_{i+2}},
\end{align*}
and for $m\ge 2$,
\begin{align*}
D_{M}(m,q)&=(q)_1^2(q^{m+1})_1\sum_{i\ge 0}\frac{q^{i^2+i(m+3)+m-1}}{(q)_i(q)_{m+i-1}}+(q)_1^2(q^{m-2})_1\sum_{i\ge 0}\frac{q^{i^2+i(m+4)+2m+2}}{(q)_i(q)_{m+i}}.
\end{align*}
\end{proposition}
We note that Proposition \ref{pro5} may can provide a alternative proof the Ji-Zang's crank inequality \eqref{JZie}, which we least to the interested reader. 

\subsection{Another approach to the inequalities for $V_d(m,n)$}
We present an alternative approach to the study of inequalities for ranks of strongly concave compositions.  Let $\left(\frac{\;\cdot\;}{\cdot}\right)$ denote the Kronecker symbol. From the proof of \cite[Proposition 1.2]{MR4001540} by Zhou, we have the following generating function
\begin{equation}\label{eq51}
\sum_{n\ge 0}V_d\left(m,n\right)q^n=\frac{q^{\binom{m+1}{2}}}{(q;q)_\infty}\sum_{\ell\ge 0}\left(\frac{-3}{2\ell+1}\right)q^{\frac{2}{3}\ell(\ell+1)+\ell m},
\end{equation}
for any $m\ge 0$.  By using \eqref{eq51}, the definition \eqref{eqda} of $D_{A}(m, q)$ and notice that the Kronecker symbol $\left(\frac{-3}{k}\right)$ is periodic in $k\in\mathbb{Z}$ with period $3$, we find that 
\begin{align}\label{eq52}
q^{-\binom{m+1}{2}}D_{V_d}(m+1,q)=&\frac{1-q^{m-{\bf 1}_{m>0}}}{(q)_\infty}\sum_{\ell\ge 0}\left(\frac{-3}{2\ell+1}\right)q^{\frac{2}{3}\ell(\ell+1)+m\ell}\nonumber\\
&+\frac{q^{m-{\bf 1}_{m>0}}}{(q)_\infty}\sum_{\ell\ge 0}\left(\frac{-3}{2\ell+1}\right)q^{\frac{2}{3}\ell(\ell+1)+m\ell}\left(1-q^{\ell+1+{\bf 1}_{m>0}}\right)\nonumber
\\
=&\frac{1}{(q^3)_\infty}\cdot\frac{1-q^{m-{\bf 1}_{m>0}}}{1-q}\sum_{\ell\ge 0}q^{6\ell^2+(3m+2)\ell}\frac{1-q^{2(4\ell+2+m)}}{1-q^2}+ q^{m-{\bf 1}_{m>0}} B(m,q),
\end{align}
for all integers $m\ge 0$, where
$$B(m,q)=\frac{1}{(q)_\infty}\sum_{\ell\ge 0}q^{6\ell^2+(3m+2)\ell}\left(1-q^{3\ell+1+{\bf 1}_{m>0}}-q^{8\ell+4+2m}+q^{11\ell+7+2m+{\bf 1}_{m>0}}\right).$$
Clearly, the coefficients in the $q$-expansion of the first term of \eqref{eq52} are already nonnegative for all $m\ge 0$. Therefore, to prove $V_d(m,n)>V_d(m+1,n)$, it suffices to study the signs of $c(m,n)$, the coefficient of $q^n$ in the $q$-expansion of $B(m,q)$.
\medskip

We now present an approach to attacking problems of this type. In view of the fact that $1/(q)_\infty$ is the generating function for the partition function $p(n)$, by convolution of power series, we see that $c(m,n)$ equals a linear combination of sums of $p(n)$ evaluated at values of quadratic polynomials. Using a version of the Hardy-Ramanujan asymptotic formula with an error term for $p(n)$, it is possible to establish an asymptotic formula with an error term for $c(m,n)$; see the proof of \cite[Theorem 1.3]{MR4001540} for an example. However, it is not easy to obtain an effective error bound using this method, because the first few leading terms in the Hardy-Ramanujan asymptotic formula are complicated. If we replace $1/(q)_\infty$ by a rational function each of whose poles is simple, then the coefficients of the $q$-expansion of this rational function can be approximated by a polynomial with a bounded error term. In this case, we find that $c(m,n)$ can be approximated by a linear combination of sums of polynomials evaluated at values of quadratic polynomials, and such sums can be computed explicitly because the sum of consecutive powers has a closed-form polynomial expression. Furthermore, the total error term can be tightly bounded, as all intermediate error terms are already well controlled. We note that this observation can be found in \cite[p.4]{MR4768277} by Zhou.

To apply the above observation, for any integer $m\ge 0$, we define
\begin{align*}
\sum_{n\ge 0}c_t(m,n)q^n=&\frac{1}{(1-q)(1-q^2)(1-q^3)(1-q^5{\bf 1}_{m=0})}\\
&\times\sum_{\ell\ge 0}q^{6\ell^2+(3m+2)\ell}\left(1-q^{3\ell+1+{\bf 1}_{m>0}}-q^{8\ell+4+2m}+q^{11\ell+7+2m+{\bf 1}_{m>0}}\right).
\end{align*}
Then, for any integer $m\ge 0$ we have
$$B(m,q)=\frac{1}{(1-q^4)(1-q^{5}{\bf 1}_{m>0})(q^6)_\infty}\sum_{n\ge 0}c_t(m,n)q^n.$$
It follows that the nonnegativity of $c_t(m,n)$ implies the positivity of $c(m,n)$. Moreover, by applying a similar argument as in Zhou \cite{MR4768277}, we can show that $c_t(m,n)>0$ for all integers $m, n\ge 0$, with the only exceptions
$$c_t(0,1)=c_t(0,4)=c_t(0,6)=c_t(0,7)=c_t(0,9)=0.$$
This result gives a proof of $V_d(m,n)>V_d(m+1,n)$ for $(m,n)\neq (0,1)$. We leave the remaining details to the interested reader.
\medskip

However, for the monotonicity of $V(m,n)$, the rank of concave compositions, a generating function like \eqref{eq51} for $V_d(m,n)$ remains to be established. 
We note that another bivariate generating function for $V(m,n)$ and $V_d(m,n)$ was found by Andrews--Rhoades--Zwegers \cite[Theorem 1.6]{MR3152010} as follows:
\begin{align*}
\sum_{\substack{n\ge 0\\ m\in\bz}}V(m,n)z^m q^{n}=\frac{q^{-1}}{(zq)_{\infty}(z^{-1})_{\infty}}\left(\frac{1}{(q)_{\infty}}\sum_{n\in\bz}\frac{q^{\frac{n(n+1)}{2}}(-z)^{-n}}{1-zq^n}-\sum_{n\ge 0}\frac{q^{n^2}}{(z)_{n+1}(z^{-1}q)_{n}}\right),
\end{align*}
and
$$
\sum_{\substack{n\ge 0\\ m\in\bz}}V_d(m,n)z^m q^{n}=(-z)_{\infty}(-z^{-1}q)_{\infty}\sum_{n\ge 0}\left(\frac{-12}{n}\right)z^{\frac{n-1}{2}}q^{\frac{n^2-1}{24}}-\sum_{n\ge 0}(-1)^nq^{\frac{n(n+1)}{2}}z^{2n+1}.
$$
We see that the bivariate generating function for $V(m,n)$ is more complicated than that of $V_d(m,n)$, which involves the inverse of the \emph{theta function}, the \emph{Appell-Lerch sums} and the \emph{universal mock theta function}. Thus it is expected that the single-variable generating function for $V(m,n)$ with $m$ given will be more complicated than \eqref{eq51}, and instead involves three-fold or higher-fold sums of \emph{partial theta functions}, and is beyond the scope of the current study.

\subsection*{Acknowledgements} This work was partially supported by Guangxi Natural Science Foundation (No. 2026GXNSFBA00640038) 
and National Natural Science Foundation of China (No. 12301423).

\begin{thebibliography}{10}

\bibitem{MR3048655}
George~E. Andrews.
\newblock Concave and convex compositions.
\newblock {\em Ramanujan J.}, 31(1-2):67--82, 2013.

\bibitem{MR929094}
George~E. Andrews and F.~G. Garvan.
\newblock Dyson's crank of a partition.
\newblock {\em Bull. Amer. Math. Soc. (N.S.)}, 18(2):167--171, 1988.

\bibitem{MR3152010}
George~E. Andrews, Robert~C. Rhoades, and Sander~P. Zwegers.
\newblock Modularity of the concave composition generating function.
\newblock {\em Algebra Number Theory}, 7(9):2103--2139, 2013.

\bibitem{MR3077150}
F.~J. Dyson.
\newblock Some guesses in the theory of partitions.
\newblock {\em Eureka}, (8):10--15, 1944.

\bibitem{MR920146}
F.~G. Garvan.
\newblock New combinatorial interpretations of {R}amanujan's partition
  congruences mod {$5,7$} and {$11$}.
\newblock {\em Trans. Amer. Math. Soc.}, 305(1):47--77, 1988.

\bibitem{MR4324846}
Kathy~Q. Ji and Wenston J.~T. Zang.
\newblock Unimodality of the {A}ndrews-{G}arvan-{D}yson cranks of partitions.
\newblock {\em Adv. Math.}, 393:Paper No. 108053, 54, 2021.

\bibitem{MR2280868}
S.~Ramanujan.
\newblock Some properties of {$p(n)$}, the number of partitions of {$n$}
  [{P}roc. {C}ambridge {P}hilos. {S}oc. {\bf 19} (1919), 207--210].
\newblock In {\em Collected papers of {S}rinivasa {R}amanujan}, pages 210--213.
  AMS Chelsea Publ., Providence, RI, 2000.

\bibitem{MR4001540}
Nian~Hong Zhou.
\newblock On the distribution of the rank statistic for strongly concave
  compositions.
\newblock {\em Bull. Aust. Math. Soc.}, 100(2):230--238, 2019.

\bibitem{MR4768277}
Nian~Hong Zhou.
\newblock Positivity and tails of pentagonal number series.
\newblock {\em J. Combin. Theory Ser. A}, 208:Paper No. 105933, 21, 2024.

\bibitem{MR4903333}
Nian~Hong Zhou.
\newblock Unimodality and certain bivariate formal {L}aurent series.
\newblock {\em European J. Combin.}, 128:Paper No. 104170, 21, 2025.

\end{thebibliography}

\end{document}